\documentclass[a4paper,10pt]{article}
\usepackage[latin1]{inputenc}
\usepackage{amsthm,amsmath,amssymb}
\usepackage{enumerate,enumitem}
\usepackage{graphicx}
\usepackage{color}

\theoremstyle{plain}
\newtheorem{thm}{Theorem}

\newtheorem{lem}[thm]{Lemma}

\newtheorem*{conjecture*}{Conjecture}

\newtheorem*{quest*}{Question}

\newcommand{\dist}{\ensuremath{{\rm dist}}}

\title{A correction of a characterization of planar partial cubes}
\author{R\'emi Desgranges\thanks{ENS Cachan, Universit\'e Paris-Saclay} \and Kolja Knauer\thanks{Laboratoire d'Informatique Fondamentale, Aix-Marseille Universit\'e and CNRS,
Facult\'e des Sciences de Luminy, F-13288 Marseille Cedex 9, France}}

\begin{document}

\maketitle

\begin{abstract}
 In this note we determine the set of expansions such that a partial cube is planar if and only if it arises by a sequence of such expansions from a single vertex. This corrects a result of Peterin.
\end{abstract}

\section{Introduction}

A graph is a \emph{partial cube} if it is isomorphic to an isometric subgraph $G$ of a hypercube graph $Q_d$, i.e., $\dist_G(v,w)=\dist_{Q_n}(v,w)$ for all $v,w\in G$. Any isometric embedding of a partial cube into a hypercube leads to the same partition of edges into so-called $\Theta$-classes, where two edges are equivalent, if they correspond to a change in the same coordinate of the hypercube. This can be shown using the Djokovi\'c-Winkler-relation $\Theta$ which is defined in the graph without reference to an embedding, see~\cite{Djo-73,Win-84}.

Let $G^1$ and $G^2$ be two isometric subgraphs of a graph $G$ that (edge-)cover $G$ and such that their intersection $G':=G^1\cap G^2$ is non-empty. The \emph{expansion} $H$ of $G$ with respect to $G^1$ and $G^2$ is obtained by considering $G^1$ and $G^2$ as two disjoint graphs and connecting them by a matching between corresponding vertices in the two resulting copies of $G'$.
A result of Chepoi~\cite{Che-88} says that a graph is a partial cube if and only if it can be obtained from a single vertex by a sequence of expansions. An equivalence class of edges with resepct to  $\Theta$ in a partial cube is an inclusion minimal edge cut. The inverse operation of an expansion in partial cubes is called \emph{contraction} and consists in taking a $\Theta$-class of edges $E_f$ and contracting it. The two disjoint copies of the corresponding $G^1$ and $G^2$ are just the two components of the graph where $E_f$ is deleted.

\section{The flaw and the result}

Let $H$ be an expansion of a planar graph $G$ with respect to $G^1$ and $G^2$. Then $H$ is a \emph{2-face expansion} of $G$ if $G^1$ and $G^2$ have plane embeddings such that $G':=G^1\cap G^2$ lies on a face in both the respective embeddings.
Peterin~\cite{Pet-08} proposes a theorem stating that a graph is a planar partial cube if and only if it can be obtained from a single vertex by a sequence of 2-face expansions. However, his argument has a flaw, since $G'$ lying on a face of $G^1$ and $G^2$ does not guarantee that the expansion $H$ be planar. 
Indeed, Figure~\ref{fig:cxmpl} shows an example of such a 2-face expansion $H$  of a planar graph $G$ that is non-planar.

 \begin{figure}[htb]
  \centering
  \includegraphics{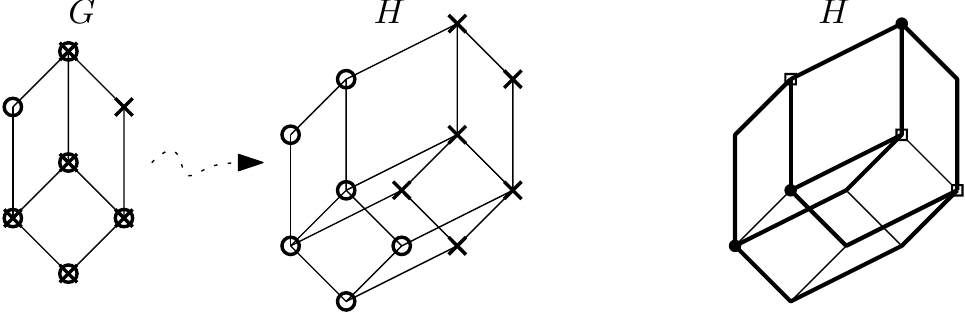}
  \caption{Left: A 2-face expansion $H$ of a planar partial cube $G$, where $G^1$ and $G^2$ are drawn as crosses and circles, respectively. Right: A subdivision of $K_{3,3}$ (bold) in $H$, certifying that $H$ is not planar.}
  \label{fig:cxmpl}
 \end{figure}

The correct concept are non-crossing 2-face expansions: We call an expansion $H$ of a planar graph $G$ with respect to subgraphs $G^1$ and $G^2$ a \emph{non-crossing 2-face expansion} if $G^1$ and $G^2$ have plane embeddings such that $G':=G^1\cap G^2$ lies on the outer face of both the respective embeddings, such that the orderings on $G'$ obtained from traversing the outer faces of $G^1$ and $G^2$ in the clockwise order, respectively, are opposite.

\begin{lem}\label{lem}
 For a partial cube $H\not\cong K_1$ the following are equivalent:
 \begin{itemize}
  \item[(i)] $H$ is planar,
  \item[(ii)] $H$ is a non-crossing 2-face expansion of a planar partial cube $G$,
  \item[(iii)] if $H$ is an expansion of $G$, then $G$ is planar and $H$ is a non-crossing 2-face expansion of $G$.
 \end{itemize}

\end{lem}
\begin{proof}~

\noindent (ii)$\Longrightarrow$(i):
Let $G$ be a planar partial cube and $G^1$ and $G^2$ two subgraphs satisfying the preconditions for doing a non-crossing 2-face expansion. We can thus embed $G^1$ and $G^2$ disjointly into the plane such that the two copies of $G':=G^1\cap G^2$ appear in opposite order around their outer face, respectively. Connecting corresponding vertices of the two copies of $G'$ by a matching $E_f$ does not create crossings, because the 2-face expansion is non-crossing, see Figure~\ref{fig:prf}.
Thus, if $H$ is a non-crossing 2-face expansion of $G$, then $H$ is planar.
 
 \begin{figure}[htb]
  \centering
  \includegraphics{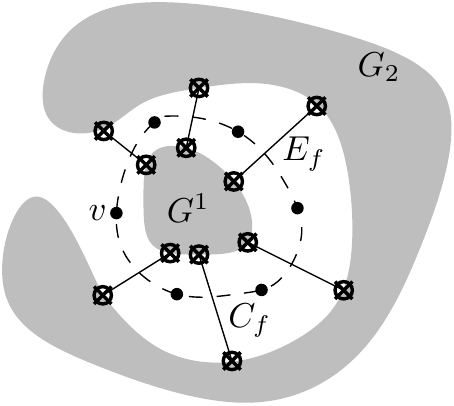}
  \caption{Two disjoint copies of subgraphs $G^1$ and $G^2$ in a planar partial cube~$H$.}
  \label{fig:prf}
 \end{figure}

 \medskip
 
\noindent (i)$\Longrightarrow$(iii): Let $H$ be a planar partial cube, that is an expansion of $G$. Thus, there is a $\Theta$-class $E_f$ of $H$ such that $G=H/E_f$. In particular, since contraction preserves planarity, $G$ is planar. 

Consider now $H$ with some planar embedding. Since $H$ is a partial cube, $E_f$ is an inclusion-minimal edge cut of $H$. Thus, $H\setminus E_f$ has precisely two components corresponding to $G^1$ and $G^2$, respectively. Since $E_f$ is a minimal cut its planar dual is a simple cycle $C_f$. It is well-known, that any face of a planar embedded graph can be chosen to be the outer face without changing the combinatorics of the embedding. We change the embedding of $H$, such that some vertex $v$ of $C_f$ corresponds to the outer face of the embedding, see Figure~\ref{fig:prf}.

Now, without loss of generality $C_f$ has $G^1$ and $G^2$ in its interior and exterior, respectively. Since $C_f$ is connected and disjoint from $G^1$ and $G^2$ it lies in a face of both. By the choice of the embedding of $H$ it is their outer face. Moreover, since every vertex from a copy of $G'$ in $G^1$ can be connected by an edge of $E_f$ to its partner in $G^2$ crossing an edge of $C_f$ but without introducing a crossing in $H$, the copies of $G'$ in $G^1$ and $G^2$ lie on this face, respectively.

Furthermore, following $E_f$ in the sense of clockwise traversal of $C_f$ gives the same order on the two copies of $G'$, corresponding to a clockwise traversal on the outer face of $G^1$ and a counter-clockwise traversal on the outer face of $G^2$. Thus, traversing both outer faces in clockwise order the obtained orders on the copies of $G'$ are opposite. Hence $H$ is a non-crossing 2-face expansion of $G$.

\medskip

\noindent (iii)$\Longrightarrow$(ii): Since $H\not\cong K_1$, it is an expansion of some partial cube $G$. The rest is trivial.
\end{proof}

Lemma~\ref{lem} yields our characterization.

\begin{thm}
 A graph $H$ is a planar partial cube if and only if $H$ arises from a sequence of non-crossing 2-face expansions from $K_1$.
\end{thm}
\begin{proof}~

\noindent $\Longrightarrow$: Since $H$ is a partial cube by the result of Chepoi~\cite{Che-88} it arises from a sequence of expansions from $K_1$. Moreover, all these sequences have the same length correspodning to the number of $\Theta$-classes of $H$. We proceed by induction on the length $\ell$ of such a sequence. If $\ell=0$ the sequence is empty and there is nothing to show. Otherwise, since $H\not\cong K_1$ is planar we can apply Lemma~\ref{lem} to get that $H$ arises by a non-crossing 2-face expansions from a planar partial cube $G$. The latter has a sequence of expansions from $K_1$ of length $\ell-1$ which by induction can be chosen to consist of non-crossing 2-face expansions. Together with the expansion from $G$ to $H$ this gives the claimed sequence from $H$.

\noindent $\Longleftarrow$: Again we induct on the length $\ell$ of the sequence. If $\ell=0$ we are fine since $K_1$ is planar. Otherwise, consider the graph $G$ in the sequence such that $H$ is its non-crossing 2-face expansion. Then $G$ is planar by induction and $H$ is planar by Lemma~\ref{lem}, since it is a non-crossing 2-face expansion of $G$.
\end{proof}

\section{Remarks}

We have characterized planar partial cubes graphs by expansions. Planar partial cubes have also been characterized in a topological way as dual graphs of non-separating pseudodisc arrangements~\cite{Alb-16}. There is a third interesting way of characterizing them. The class of planar partial cubes is closed under \emph{partial cube minors}, see~\cite{Che-16}, i.e., contraction of $G$ to $G/E_f$ where $E_f$ is a $\Theta$-class and restriction to a component of $G\setminus E_f$. What is the family of minimal obstructions for a partial cube to being planar, with respect to this notion of minor? The answer will be an infinite list, since a subfamily is given by the set $\{G_n\square K_2\mid n\geq 3\}$, where $G_n$ denotes the \emph{gear graph} (also known as \emph{cogwheel}) on $2n+1$ vertices and $\square$ is the Cartesian product of graphs. See Figure~\ref{fig:obstr} for the first three members of the family.

 \begin{figure}[htb]
  \centering
  \includegraphics{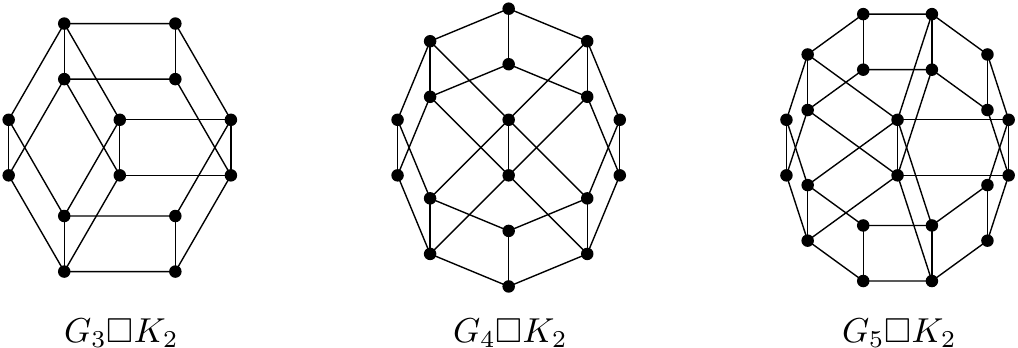}
  \caption{The first three members of an infinite family of minimal obstructions for planar partial cubes.}
  \label{fig:obstr}
 \end{figure}

\subsubsection*{Acknowledgements:} We wish to thank Iztok Peterin for discussing this result with us and two referees for helpful comments.

\bibliography{plit}

\begin{thebibliography}{1}

\bibitem{Alb-16}
{\sc M.~{Albenque} and K.~{Knauer}}, {\em {Convexity in partial cubes: the hull
  number.}}, {Discrete Math.}, 339 (2016), pp.~866--876.

\bibitem{Che-16}
{\sc V.~{Chepoi}, K.~{Knauer}, and T.~{Marc}}, {\em {Partial cubes without
  $Q\_3^-$ minors}},  (2016).
\newblock arXiv:1606.02154.

\bibitem{Che-88}
{\sc V.~{Chepoj}}, {\em {Isometric subgraphs of Hamming graphs and
  $d$-convexity.}}, {Cybernetics}, 24 (1988), pp.~6--11.

\bibitem{Pet-08}
{\sc I.~{Peterin}}, {\em {A characterization of planar partial cubes.}},
  {Discrete Math.}, 308 (2008), pp.~6596--6600.

\bibitem{Djo-73}
{\sc D.~\v{Z}. {Djokovi\'c}}, {\em {Distance-preserving subgraphs of
  hypercubes.}}, {J. Comb. Theory, Ser. B}, 14 (1973), pp.~263--267.

\bibitem{Win-84}
{\sc P.~M. {Winkler}}, {\em {Isometric embedding in products of complete
  graphs.}}, {Discrete Appl. Math.}, 7 (1984), pp.~221--225.

\end{thebibliography}
\bibliographystyle{my-siam}
 
\end{document}